\documentclass{amsart}
\usepackage[utf8]{inputenc}

\usepackage{hyperref}
\usepackage{amsmath}
\usepackage{amsmath,amsthm,amssymb,amsfonts}
\usepackage[utf8]{inputenc}
\usepackage{enumerate}
\usepackage{graphicx}
\usepackage{enumitem}

\usepackage{tikz}
\usepackage[numbers]{natbib}

\theoremstyle{plain}
\newtheorem{theorem}{Theorem}[section]
\newtheorem{lemma}[theorem]{Lemma}
\newtheorem{claim}[theorem]{Claim}
\newtheorem{corollary}[theorem]{Corollary}
\newtheorem{proposition}[theorem]{Proposition}

\newtheorem{question}[theorem]{Question}
\newtheorem{problem}[theorem]{Problem}
\newtheorem{fact}[theorem]{Fact}
\newcommand{\thistheoremname}{}
\newtheorem*{genericthm*}{\thistheoremname}
\newenvironment{namedthm*}[1]
{\renewcommand{\thistheoremname}{#1}%
	\begin{genericthm*}}
	{\end{genericthm*}}

\theoremstyle{definition}\newtheorem{definition}[theorem]{Definition}

\theoremstyle{definition}
\theoremstyle{definition}\newtheorem{remark}[theorem]{Remark}

\newcommand{\mc}{\mathcal}

\newcommand{\Q}{\mathbb{Q}}

\newcommand{\N}{\mathbb{N}}
\newcommand{\Z}{\mathbb{Z}}
\DeclareMathOperator{\hp}{Homeo^+}

\DeclareMathOperator{\fix}{Fix}

\DeclareMathOperator{\aut}{Aut}
\DeclareMathOperator{\id}{id}
\DeclareMathOperator{\conv}{conv}

\DeclareMathOperator{\ACL}{ACL}
\DeclareMathOperator{\dom}{dom}
\DeclareMathOperator{\ran}{ran}
\newcommand{\homeo}{Homeo}

\DeclareMathOperator{\proj}{proj}

\newcommand{\nr}{\lnot R}

\begin{document}
	
	\title{The structure of random automorphisms of countable structures}
	\author[U. B. Darji]{Udayan B. Darji}
	\address{Department of Mathematics, University of Louisville,
		Louisville, KY 40292, USA\\Ashoka University, Rajiv Gandhi Education City, Kundli, Rai 131029, India} 
	\email{ubdarj01@louisville.edu}
	\urladdr{http://www.math.louisville.edu/\!$\tilde{}$ \!\!darji}
	\author[M. Elekes]{M\'arton Elekes}
	\address{Alfr\'ed R\'enyi Institute of Mathematics, Hungarian Academy of Sciences,
		PO Box 127, 1364 Budapest, Hungary and E\"otv\"os Lor\'and
		University, Institute of Mathematics, P\'azm\'any P\'eter s. 1/c,
		1117 Budapest, Hungary}
	\email{elekes.marton@renyi.mta.hu}
	\urladdr{www.renyi.hu/ \!$\tilde{}$ \!\!emarci}
	
	\author[K. Kalina]{Kende Kalina}
	\address{E\"otv\"os Lor\'and
		University, Institute of Mathematics, P\'azm\'any P\'eter s. 1/c,
		1117 Budapest, Hungary}
	\email{kkalina@cs.elte.hu}
	\author[V. Kiss]{Viktor Kiss}
	\address{Alfr\'ed R\'enyi Institute of Mathematics, Hungarian Academy of Sciences,
		PO Box 127, 1364 Budapest, Hungary and E\"otv\"os Lor\'and
		University, Institute of Mathematics, P\'azm\'any P\'eter s. 1/c,
		1117 Budapest, Hungary}
	\email{kiss.viktor@renyi.mta.hu}
	
	\author[Z. Vidny\'anszky]{Zolt\'an Vidny\'anszky}
	\address{Kurt G\"{o}del Research Center for Mathematical Logic, 
		Universit\"{a}t Wien,
		W\"{a}hringer Stra{\ss}e 25, 
		1090 Wien,
		Austria and Alfr\'ed R\'enyi Institute of Mathematics, Hungarian Academy of Sciences,
		PO Box 127, 1364 Budapest, Hungary}
	\email{zoltan.vidnyanszky@univie.ac.at}
	\urladdr{
		http://www.logic.univie.ac.at/~vidnyanszz77/
	}
	
	\thanks{The second, fourth and fifth authors were partially supported by the
		National Research, Development and Innovation Office
		-- NKFIH, grants no.~113047, no.~104178 and no. ~124749. The fifth
		author was also supported by FWF Grant P29999.}
	\begin{abstract}
		In order to understand the structure of the ``typical'' element of an automorphism group, one has to study how large the conjugacy classes of the group are. When typical is meant in the sense of Baire category, a complete description of the size of the conjugacy classes has been given by Kechris and Rosendal. Following Dougherty and Mycielski we investigate the measure theoretic dual of this problem, using Christensen's notion of Haar null sets. When typical means random, that is, almost every with respect to this notion of Haar null sets, the behavior of the automorphisms is entirely different from the Baire category case. 
		In this paper, we generalize the theorems of Dougherty and Mycielski about $S_\infty$ to arbitrary automorphism groups of countable structures isolating a new model theoretic property, the \emph{Cofinal Strong Amalgamation Property}.  As an application we show that a large class of automorphism groups can be decomposed into the union of a meager and a Haar null set.

	\end{abstract}
	\maketitle
	
	\section{Introduction}

	The study of typical elements of Polish groups is a flourishing field with a 
	large number of applications. The systematic investigation of typical elements of automorphism groups of countable structures was initiated by Truss \cite{truss1992generic}. He conjectured that the existence of a co-meager conjugacy class can be characterized in model theoretic terms. This question was answered affirmatively by Kechris and Rosendal \cite{KechrisRosendal}. They, extending the work of Hodges, Hodkinson, Lascar and Shelah \cite{shelah} also investigated the relation between the existence of co-meager conjugacy classes in every dimension and other group theoretic properties, such as the small index property, uncountable cofinality, automatic continuity and Bergman's property. 
	
	The existence and description of typical elements frequently have applications in the theory of dynamical systems as well. For example, it is easy to see that the automorphism group of the countably infinite atomless Boolean algebra is isomorphic to the homeomorphism group of the Cantor set, which is a central object in dynamics. Thus, from their general results Kechris and Rosendal deduced the existence of a co-meager conjugacy class in the homeomorphism group of the Cantor set. A description of an element with such a class was given by Glasner and Weiss \cite{glasner2003universal}  and from a different perspective by Bernardes and the first author \cite{bernardes}.
	
	Thus, it is natural to ask 
	whether there exist measure theoretic analogues of these results. 
	Unfortunately, on non-locally compact groups there is no natural invariant 
	$\sigma$-finite measure. However, a generalization of the ideal of measure zero 
	sets can be defined in every Polish group as follows:
	
	\begin{definition}[Christensen, \cite{originalhaarnull}]
		\label{d:haarnull}
		Let $G$ be a 
		Polish group and $B \subset G$ be Borel. We say that $B$ is \textit{Haar null} if 
		there exists a
		Borel probability measure $\mu$ on $G$ such that for every $g,h \in G$
		we have $\mu(gBh)=0$. 
		An arbitrary set $S$ is called Haar null if $S \subset B$ for some Borel Haar 
		null set $B$.
	\end{definition}
	
	It is known that the collection of Haar null sets forms a $\sigma$-ideal in 
	every Polish group (see \cite{cohen} and \cite{mycorig}) and it coincides with the ideal of measure zero sets in 
	locally compact groups with respect to every left (or equivalently right) Haar measure. 
	Using this definition, it makes sense to talk about the properties of random 
	elements of a Polish group. A property $P$ of elements of a Polish group $G$ is said 
	to \textit{hold almost surely} or \emph{almost every element of G has property 
		$P$} if the set
	$\{g \in G: g \text{ has property } P\}$ is co-Haar null.

	Since we are primarily interested in homeomorphism and automorphism groups, and in such groups conjugate elements can be considered isomorphic, we are only 
	interested in the conjugacy invariant properties of the elements of our Polish groups. Hence, in order to describe  the random element, 
	one must give a complete description of the size of the conjugacy classes with 
	respect to the Haar null ideal. The investigation of this question has been 
	started by Dougherty and Mycielski \cite{DM} in the permutation group of a 
	countably infinite set, $S_\infty$. If $f \in S_\infty$ and $a$ is an element 
	of the underlying set then the set $\{f^{k}(a):k \in \mathbb{Z}\}$ is called 
	the \textit{orbit of $a$ (under $f$)}, while the cardinality of this set is 
	called \textit{orbit length}. Thus, each $f \in S_\infty$ has a collection of 
	orbits (associated to the elements of the underlying set). It is easy to 
	show that two elements of $S_\infty$ are conjugate if and only if they have the 
	same (possibly infinite) number of orbits for each possible orbit length. 
	
	\begin{theorem}[Dougherty, Mycielski, \cite{DM}] Almost every element of 
		$S_\infty$ has 
		infinitely many infinite orbits and only finitely many finite ones.
	\end{theorem}
	Therefore, almost all permutations belong to the union of a countable set of 
	conjugacy classes.
	\begin{theorem}[Dougherty, Mycielski, \cite{DM}] 
		\label{t:DMold} All of these countably many conjugacy classes are non-Haar null.
	\end{theorem}
	
	Thus, the above theorems give a complete description of the non-Haar null 
	conjugacy classes and the (conjugacy invariant) properties of a random element. The aim of our paper is to initiate a systematic study of the size of the conjugacy classes of automorphism groups of countable structures. Our work is centered around questions of the following type:
	
	\begin{question}
		Let  $\mathcal{A}$ be a countable (first order) structure.
		\begin{enumerate}
			\item \label{q:darji} What properties of $\mathcal{A}$ ensure that (an appropriate) generalization of the theorem of Dougherty and Mycielski holds for $\aut(\mathcal{A})$?
			\item \label{q:general} Describe the (conjugacy invariant) properties of almost every element of $\aut(\mathcal{A})$: Which conjugacy classes of $\aut(\mathcal{A})$ are non-Haar null? How many non-Haar null conjugacy classes are there? Is almost every element of $\aut(\mathcal{A})$ contained in a non-Haar null class?
		\end{enumerate}
	\end{question}

	 In this paper we answer the first question, see Section \ref{s:results} and Theorem \ref{t:gen}.
	
	One can prove that in $S_\infty$ the collection of elements that have no infinite orbits is a co-meager set. This shows that the typical behavior in the sense of Baire category is quite different from the typical behavior in the measure theoretic sense. In particular, $S_\infty$ can be decomposed into the union of a Haar null and a meager set. It is well known that this is possible in every locally compact group, but the situation is not clear in the non-locally compact case. Thus, the below question of the first author arises naturally:
	
	\begin{question}
		Suppose that $G$ is an uncountable Polish group. Can it be written as the 
		union of a meager and a Haar null set?
	\end{question}
	
	We investigate this question for various automorphism groups, and solve it for a large class, see Corollary \ref{c:inffixed -> decomp}.

	The paper is organized as follows. First, in Section  \ref{s:prel} we summarize facts and notations used later, then in Section \ref{s:results} we give a detailed description of our results. For the sake of the transparency of the topic we also include in this section the results of two upcoming papers \cite{autqcikk} and \cite{autrcikk}. Section \ref{s:genres} contains our main theorem, while in Section \ref{s:appl} we apply the general result to prove a theorem about Haar null-meager decompositions. After this, in Section \ref{s:poss} we investigate the possible cardinality of non-Haar null conjugacy classes of (locally compact and non-locally compact) Polish groups. Finally, we conclude with listing a number of open questions in Section \ref{s:open}.

	\section{Preliminaries and notations}
	\label{s:prel}
	We will follow the notations of \cite{kechrisbook}. For a detailed introduction to the theory of Polish groups see \cite[Chapter 1]{becker1996descriptive}, while the model theoretic background can be found in \cite[Chapter 7]{hodges}. Nevertheless, we summarize the basic facts which we will use. 
	
	As usual, a countable structure $\mathcal{A}$ is a first order structure on a countable set with countably many constants, relations and functions. The underlying set will be denoted by $\dom(\mc{A})$. The automorphism group of the structure $\mathcal{A}$ is denoted by $\aut(\mathcal{A})$ which we consider as a topological (Polish) group with the topology of pointwise convergence. Isomorphisms between topological groups are considered to be group automorphisms that are also homeomorphisms. 
	The structure $\mathcal{A}$ is called \textit{ultrahomogeneous} if every isomorphism between its finitely generated substructures extends to an automorphism of $\mathcal{A}$. The \textit{age} of a structure $\mathcal{A}$ is the collection of the finitely generated substructures of $\mathcal{A}$. An injective homomorphism between structures will be called an \emph{embedding}. A structure is said to be \emph{locally finite} if every finite set of elements generates a finite substructure.  
	
	A countable set $\mc{K}$ of finitely generated structures  of the same language is called a \textit{Fra\"iss\'e class} if it satisfies the hereditary (HP), joint embedding (JEP) and amalgamation properties (AP) (see \cite[Chapter 7]{hodges}). We will need the notion of the strong amalgamation property: A Fra\"iss\'e class $\mc{K}$ satisfies the \emph{strong amalgamation property (SAP)} if  for every $\mc{B} \in \mc{K}$ and every pair of structures $\mc{C}, \mc{D} \in \mc{K}$ and embeddings $\psi: \mc{B} \to \mc{C}$ and $\chi:\mc{B} \to \mc{D}$ there exist $\mc{E} \in \mc{K}$ and embeddings $\psi': \mc{C} \to \mc{E}$ and $\chi':\mc{D} \to \mc{E}$ such that \[\psi' \circ \psi=\chi'\circ \chi \text{ and } \psi'(\mc{C})\cap \chi'(\mc{D})= (\psi' \circ \psi)(\mc{B})=(\chi'\circ \chi) (\mc{B}).\]
	
	For a Fra\"iss\'e class $\mc{K}$ the unique countable ultrahomogeneous structure $\mc{A}$ with $age(\mc{A})=\mc{K}$ is called the \emph{Fra\"iss\'e limit of} $\mc{K}$.
	If $G$ is the automorphism group of a structure $\mathcal{A}$, we call a bijection $p$ a \emph{partial automorphism} or a \emph{partial permutation} if it is an automorphism between two finitely generated substructures of $\mathcal{A}$ such that $p \subset g$ for some $g \in G$.
	
	As mentioned before, $S_\infty$ stands for the permutation group of the countably infinite set $\omega$. It is well known that $S_\infty$ is a Polish group with the pointwise convergence topology. This coincides with the topology generated by the sets of the form $[p]=\{f \in S_\infty: p \subset f\}$, where $p$ is a finite partial permutation.
	
	Let $\mathcal{A}$ be a countable structure. By the countability of $\mc{A}$, every automorphism $f \in \aut(\mc{A})$ can be regarded as an element of $S_\infty$, and it is not hard to see that in fact $\aut(\mc{A})$ will be a closed subgroup of $S_\infty$. Moreover, the converse is also true, namely every closed subgroup of $S_\infty$ is isomorphic to the automorphism group of a countable structure.

	Let $G$ be a closed subgroup of $S_\infty$. The \emph{orbit} of an element $x \in 
	\omega$ (under $G$) is the set $G(x) = \{y \in \omega : \exists g \in G \; (g(x) 
	= y)\}$. 
	For a set $S \subset \omega$ we denote the \emph{pointwise stabiliser} of $S$ 
	by $G_{(S)}$, that is, $G_{(S)} = \{g \in G : \forall s \in S\; (g(s) = s)\}$. In 
	case $S = \{x\}$, we write $G_{(x)}$ instead of $G_{(\{x\})}$. 
	
	As in the case of $S_\infty$, for a countable structure $\mathcal{A}$, an element $a \in \dom(\mc{A})$ and $f \in \aut(\mc{A})$ the set $\{f^{k}(a):k \in \mathbb{Z}\}$ is called the \textit{orbit of $a$} and denoted by $\mc{O}^f(a)$, while the cardinality of this set is called \textit{orbit length}. The \textit{collection of the orbits of $f$}, or \emph{the orbits of $f$} is the set $\{\mc{O}^f(a):a \in \dom(\mc{A})\}$. If $S \subset \dom(\mc{A})$ we will also use the notation $\mathcal{O}^f(S)$ for the set $\bigcup_{a \in S}\mc{O}^f(a)$.

	We will constantly use the following fact.
	
	\begin{fact}
		\label{f:compactchar} Let $\mathcal{A}$ be a countable structure. A closed subset $C$ of $\aut(\mathcal{A})$ is compact if and only if for every $a \in \dom(\mc{A})$ the set $\{f(a),f^{-1}(a):f \in C\}$ is finite.
	\end{fact}
	
	We denote by $\mathcal{B}_\infty$ the countable atomless Boolean algebra, by $(\mathbb{Q},<)$ or $\mathbb{Q}$ the rational numbers as an ordered set. Let us use the notation $\mathcal{R}$ (or $(V,R)$) for the countably infinite random graph, that is, the unique countable graph with the following property: for every pair of finite disjoint sets $A,B \subset V$ there exists $v \in V$ such that $(\forall x \in A)(x R v)$ and $(\forall y \in B)(y \nr v)$.

	Let us consider the following notion of largeness:
	
	\begin{definition}
		\label{d:catcherbiter}
		Let $G$ be a Polish topological group. A set $A \subset G$ is called \textit{compact catcher} if for every compact $K \subset G$ there exist $g,h \in G$ so that $gKh \subset A$.  $A$ is \textit{compact biter} if for every compact $K \subset G$ there exist an open set $U$ and $g,h \in G$ so that $U \cap K \not = \emptyset$, and $g(U \cap K)h \subset A$. 
	\end{definition}
	The following easy observation is one of the most useful tools to prove that a certain set is not Haar null.
	\begin{fact}
		\label{f:biter}
		If $A$ is compact biter then it is not Haar null.
	\end{fact}
	\begin{proof}
		Suppose that this is not the case and let $B \supset A$ be a Borel Haar null set and $\mu$ be a witness measure for $B$. Then, by the regularity of $\mu$, there exists a compact set $K \subset G$ such that $\mu(K)>0$. Subtracting the relatively open $\mu$ measure zero subsets of $K$ we can suppose that for every open set $U$ if $U \cap K \not =\emptyset$ then $\mu(U \cap K)>0$. But, as $A$ is compact biter, so is $B$, thus for some open set $U$ with $\mu(U \cap K)>0$ there exist $g,h \in G$ so that $g(U\cap K)h \subset B$. This shows that $\mu$ cannot witness that $B$ is Haar null, a contradiction.
	\end{proof}
	
	Note that the proof of Theorem \ref{t:DMold} by Dougherty and Mycielski actually shows that every non-Haar null conjugacy class is compact biter and the unique non-Haar null conjugacy class which contains elements without finite orbits is compact catcher.
	
	It is sometimes useful to consider right and left Haar null sets: a Borel set $B$ is \emph{right (resp. left) Haar null} if there exists a Borel probability measure $\mu$ on $G$ such that for every $g \in G$
	we have $\mu(Bg)=0$ (resp. $\mu(gB)=0$). An arbitrary set $S$ is called \emph{right (resp. left) Haar null} if $S \subset B$ for some Borel right (resp. left) Haar null set $B$. The following observation will be used several times.

	\begin{lemma}
		\label{l:conjugacyinvariant}
		Suppose that $B$ is a Borel set that is invariant under conjugacy. Then $B$ is left Haar null iff it is right Haar null iff it is Haar null.
	\end{lemma}
	\begin{proof}
		Let $\mu$ be a measure witnessing that $B$ is left Haar null. We check that it also witnesses the Haar nullness of $B$. Indeed, let $g,h \in G$ arbitrary, $\mu(gBh)=\mu(ghh^{-1}Bh)=\mu(ghB)=0$. The proof is analogous when $B$ is right Haar null.
		
	\end{proof}

	\section{Description of the results}
	\label{s:results}

	We start with defining the crucial notion for the description of the orbits of a random element of an automorphism group. Informally, the following definition says that our structure is free enough: if we want to extend a partial automorphism defined on a finite set, there are only finitely many points for which we have only finitely many options.

	\begin{definition} \label{NACdef} Let $G$ be a closed subgroup of $S_\infty$. We say that \textit{$G$ has the finite algebraic closure property ($FACP$)} if for every finite $S \subset \omega$ the set $\{b:|G_{(S)}(b)|<\infty\}$ is finite.
	\end{definition}
	
	The following model theoretic property of Fra\"iss\'e classes turns out to be essentially a reformulation of the $FACP$ for the automorphism groups of the limits.

	\begin{definition} \label{CSAPdef} Let $\mc{K}$ be a Fra\"iss\'e class. We say that \textit{$\mc{K}$ has the cofinal strong amalgamation property (CSAP)} if there exists a subclass of $\mc{K}$ cofinal under embeddability, which satisfies the strong amalgamation property, or more formally: for every $\mc{B}_0 \in \mc{K}$ there exists a $\mc{B} \in \mc{K}$ and an embedding $\phi_0:\mc{B}_0 \to \mc{B}$ so that the \emph{strong amalgamation property holds over $\mc{B}$}, that is, for every pair of structures $\mc{C}, \mc{D} \in \mc{K}$ and embeddings $\psi: \mc{B} \to \mc{C}$ and $\chi:\mc{B} \to \mc{D}$ there exist $\mc{E} \in \mc{K}$ and embeddings $\psi': \mc{C} \to \mc{E}$ and $\chi':\mc{D} \to \mc{E}$ such that \[\psi' \circ \psi=\chi'\circ \chi \text{ and } \psi'(\mc{C})\cap \chi'(\mc{D})= (\psi' \circ \psi)(\mc{B})=(\chi'\circ \chi) (\mc{B}).\]

		A Fra\"iss\'e limit $\mc{A}$ is said to \textit{have the cofinal strong amalgamation property} if $age(\mc{A})$ has the CSAP.	
	\end{definition}

	Generalizing the results of Dougherty and Mycielski we show that the $FACP$ is equivalent to some properties of the orbit structure of a random element of the group.

	\begin{namedthm*}{Theorem \ref{t:gen}}
		Let $\mathcal{A}$ be a locally finite Fra\"iss\'e limit. Then the following are equivalent:
		\begin{enumerate}
			\item \label{tc:finfin} almost every element of $\aut(\mc{A})$ has finitely many finite orbits,
			\item \label{tc:FACP} $\aut(\mc{A})$ has the $FACP$,
			\item \label{tc:CSAP} $\mc{A}$ has the CSAP.
		\end{enumerate}
		Moreover, any of the above conditions implies that almost every element of $\mc{A}$ has infinitely many infinite orbits.
	\end{namedthm*}
	Note that every relational structure and also $\mc{B}_\infty$ is locally finite, moreover, it is well known that the ages of the structures $\mathcal{R}, (\mathbb{Q},<)$ and $\mc{B}_\infty$ have the strong amalgamation property which clearly implies the CSAP (it is also easy to directly check the $FACP$ for these groups). Hence we obtain the following corollary.
	\begin{corollary}
		\label{c:fininf}
		In $\aut(\mathcal{R}), \aut(\mathbb{Q}, <)$ and $\aut(\mathcal{B}_\infty)$ almost every element has finitely many finite and infinitely many infinite orbits.
	\end{corollary}

	As a corollary of our results, in Section \ref{s:appl} we show that a large number of groups can be partitioned in a Haar null and a meager set. 

	\begin{namedthm*}{Corollary \ref{c:inffixed -> decomp}}
		Let $G$ be a closed subgroup of $S_\infty$ satisfying the $FACP$ and suppose that the set $F = \{g \in G : \text{$\fix(g)$ is infinite}\}$ is dense in $G$. Then $G$ can be decomposed into the union of an (even conjugacy invariant) Haar null and a meager set. 
	\end{namedthm*}
	
	\begin{namedthm*}{Corollary \ref{c:Aut(R), Aut(Q), Aut(B_inf) decomposes}}
		$\aut(\mathcal{R})$, $\aut(\mathbb{Q}, <)$ and $\aut(\mc{B}_\infty)$ (and hence $\homeo(2^\mathbb{N})$) can be decomposed into the union of an (even conjugacy invariant) Haar null and a meager set.
	\end{namedthm*}

	However, these results are typically far from the full description of the behavior of the random elements. We continue with summarizing our results from \cite{autrcikk} and \cite{autqcikk} about two special cases, $\aut(\mathbb{Q}, <)$ and $\aut(\mathcal{R})$, where we gave a complete description of the Haar positive conjugacy classes. 
	
	\subsection{Summary of the random behavior in \texorpdfstring{$\aut(\mathbb{Q}, <)$}{Aut(Q, <)} and \texorpdfstring{$\aut(\mc{R})$}{Aut(R)}}

	In order to describe our results about $\aut(\Q, <)$ we need the concept of orbitals (defined below, for more details on this topic see \cite{Glass}). Let $p, q \in \Q$. The interval $(p, q)$ will denote the set $\{r \in \Q : p < r < q\}$. For an automorphism $f \in \aut(\Q, <)$, we denote the set of fixed points of $f$ 
	by $\fix(f)$. 
	\begin{definition}
		\label{d:orbital}
		The set of \emph{orbitals} of an automorphism $f \in \aut(\Q, <)$, $\mathcal{O}^*_f$, consists of 
		the convex hulls (relative to $\Q$) of the orbits of the rational numbers, that 
		is
		$$
		\mathcal{O}^*_f = \{ \conv(\{f^n(r) : n \in \Z\}) : r \in \Q\}.
		$$
	\end{definition}
	It is easy to see that the orbitals of $f$ form a partition of $\Q$, with the 
	fixed points determining one element orbitals, hence ``being in the same 
	orbital'' is an equivalence relation. Using this fact, we define the relation 
	$<$ on the set of orbitals by letting $O_1 < O_2$ for distinct $O_1, O_2 \in 
	\mathcal{O}^*_f$ if $p_1 < p_2$ for some (and hence for all) $p_1 \in O_1$ and 
	$p_2 \in O_2$. Note 
	that $<$ is a linear order on the set of orbitals. 
	
	It is also easy to see that if $p, q \in \Q$ are in the same orbital of $f$ 
	then $f(p) > p \Leftrightarrow f(q) > q$, $f(p) < p \Leftrightarrow f(q) < q$ 
	and $f(p) = p \Leftrightarrow f(q) = q \Rightarrow p = q$. This observation 
	makes it possible to define the \emph{parity function}, $s_f : \mathcal{O}^*_f 
	\to \{-1, 0, 1\}$. Let $s_f(O) = 0$ if $O$ consists of a fixed point of $f$, 
	$s_f(O) = 1$ if $f(p) > p$ for some (and hence, for all) $p \in O$ and $s_f(O) 
	= -1$ if $f(p) < p$ for some (and hence, for all) $p \in O$.  
	
	\begin{theorem} \label{t:qintro} (see \cite{autqcikk})

		For almost every element $f$ of $\aut(\mathbb{Q}, <)$ 
		\begin{enumerate}
			\item for distinct orbitals $O_1, O_2 \in \mathcal{O}^*_f$ (see Definition \ref{d:orbital})
			with $O_1 < O_2$ such that $s_f(O_1) = s_f(O_2) = 1$ or $s_f(O_1) = s_f(O_2) 
			= -1$, there exists an orbital $O_3 \in \mathcal{O}^*_f$ with $O_1 < O_3 < O_2$ 
			and $s_f(O_3) \neq s_f(O_1)$,
			\item (follows from Theorem \ref{t:gen}) $f$ has only finitely many fixed points. 
		\end{enumerate}
		These properties characterize the non-Haar null conjugacy classes, i. e., a conjugacy class is non-Haar null if and only if one (or equivalently each) of its elements has these properties.
		
		Moreover, every non-Haar null conjugacy class is compact biter and those non-Haar null classes in which the elements have no rational fixed points are compact catchers. 
	\end{theorem}
	
	This yields the following surprising corollary (for the details see \cite{autqcikk}):
	
	\begin{corollary}
		\label{c:autqcont} 
		There are continuum many non-Haar null conjugacy classes in $\aut(\mathbb{Q},<)$, and their union is co-Haar null.
	\end{corollary}

	Note that it was proved by Solecki \cite{openlyhaarnull} that in every non-locally compact Polish group that admits a two-sided invariant metric there are continuum many pairwise disjoint non-Haar null Borel sets, thus the above corollary is an extension of his results for $\aut(\mathbb{Q},<)$ (see also the case of $\aut(\mc{R})$ below). We would like to point out that in a sharp contrast to this result, in $\hp([0,1])$ (that is, in the group of order preserving homeomorphisms of the interval) the random behavior is quite different (see \cite{homeo}), more similar to the case of $S_\infty$: there are only countably many non-Haar null conjugacy classes and their union is co-Haar null.

	The characterization of non-Haar null conjugacy classes of the automorphism group of the random graph appears to be similar to the characterization of the non-Haar null classes of $\aut(\Q, <)$, however their proofs are completely different. 
	
	\begin{theorem} \label{t:randomintro} (see \cite{autrcikk}) For almost every element $f$ of $\aut(\mc{R})$
		\begin{enumerate}
			\item for every pair of finite disjoint sets, $A,B \subset V$ there exists $v \in V$ such that $(\forall x \in A)( x R v)$ and $(\forall y \in B)(y \nr v)$  \textit{ and $v \not \in \mathcal{O}^f(A \cup B)$, i. e., the union of orbits of the elements of $A \cup B$},
			\item (from Theorem \ref{t:gen}) $f$ has only finitely many finite orbits. 
		\end{enumerate}
		These properties characterize the non-Haar null conjugacy classes, i. e., a conjugacy class is non-Haar null if and only if one (or equivalently each) of its elements has these properties.
		
		Moreover, every non-Haar null conjugacy class is compact biter and those non-Haar null classes in which the elements have no finite orbits are compact catchers.  
	\end{theorem}

	It is not hard to see that this characterization again yields the following corollary (see \cite{autrcikk}):
	\begin{corollary}
		\label{c:autrcont}
		There are continuum many non-Haar null classes in $\aut(\mc{R})$ and their union is co-Haar null.
	\end{corollary}

	\subsection{Various behaviors}
	Examining any Polish group we can ask the following questions:
	
	\begin{question} 
		\label{q:how}
		\begin{enumerate}
		\item How many non-Haar null conjugacy classes are there?
		\item Is the union of the Haar null conjugacy classes Haar null?
		\end{enumerate}
	\end{question}

	Note that these are interesting even in compact groups. Table \ref{tab:1} summarizes our examples and the open questions as well (the left column indicates the number of non-Haar null conjugacy classes, while C, LC $\setminus$ C and NLC stands for compact, locally compact non-compact and non-locally compact groups, respectively). HNN denotes the well known infinite group, constructed by G. Higmann, B. H. Neumann and H. Neumann \cite{higman1949embedding}, with two conjugacy classes, while $\mathbb{Q}_d$ stands for the rationals with the discrete topology. The action, $\phi$, of $\mathbb{Z}_2$ on $\mathbb{Z}^\omega_3$ and $\mathbb{Q}^\omega_d$ is the map defined by $a \mapsto -a$.
	\begin{table}[h!]
		\begin{tabular}{ |c||c|c|c|  }
			\hline
			& \multicolumn{3}{c|}{The union of the Haar null classes  is Haar null}\\ \hline
			& C & LC $\setminus$ C & NLC \\ \hline
			$0$ & \bf{--} & \bf{--} & \bf{--} \\ \hline
			$n$ & $\mathbb{Z}_n$ & HNN & $???$ \\ \hline
			$\aleph_0$ & $???$ & $\mathbb{Z}$ & $S_\infty$   \\ \hline
			$\mathfrak{c}$ & \bf{--} & \bf{--} & $\aut(\mathbb{Q}, <)$; $\aut(\mc{R})$ \\ \hline
			& \multicolumn{3}{c|}{The union of the Haar null classes  is not Haar null} \\ \hline 
			& C & LC $\setminus$ C & NLC \\ \hline
			$0$ & $2^\omega$ & $\mathbb{Z} \times 2^\omega$ & $\mathbb{Z}^\omega$ \\ \hline
			$n$ & $\mathbb{Z}_n \times(\mathbb{Z}_2 \ltimes_\phi \mathbb{Z}_3^\omega)$ & 
			HNN $\times (\mathbb{Z}_2 \ltimes_\phi \mathbb{Z}^\omega_3)$ & 
			$\mathbb{Z}_n \times (\mathbb{Z}_2 \ltimes_\phi \mathbb{Q}_d^\omega)$ \\ \hline
			$\aleph_0$ & $???$ & $\mathbb{Z} \times( \mathbb{Z}_2 \ltimes_\phi \mathbb{Z}_3^\omega)$ & 
			$S_\infty \times (\mathbb{Z}_2\ltimes_\phi \mathbb{Z}_3^\omega)$ \\ \hline
			$\mathfrak{c}$ & \bf{--} & \bf{--} & $\aut(\mathbb{Q}, <) \times (\mathbb{Z}_2 \ltimes_\phi \mathbb{Z}_3^\omega)$
			\\ \hline
			
		\end{tabular}
		\vspace{5mm}
		\caption{Examples of various behaviors}
		\label{tab:1}
	\end{table}
	
	\section{Main results}
	\label{s:genres}
	
	This section contains our generalization of the result of Dougherty and Mycielski to automorphism groups of countable structures. For the sake of simplicity we will use the following notation.
	
	\begin{definition}
		Let $G$ be a closed subgroup of $S_\infty$ and let $S \subset \omega$ be a 
		finite subset. The \emph{group-theoretic algebraic closure} of $S$ is:
		$$
		\ACL(S)=\{x \in \omega : \text{the orbit of $x$ under $G_{(S)}$ is finite}\}.
		$$
		
	\end{definition}
	
	Obviously $G$ has the finite algebraic closure property (see Definition \ref{NACdef}) if and only if for every finite set $S$ the set $\ACL(S)$ is finite. We start with proving a simple observation about the operator $\ACL$.

	\begin{lemma}\label{l:idempotentness}
		If a group $G$ has the $FACP$ then the corresponding operator $\ACL$ is 
		idempotent.
	\end{lemma}
	\begin{proof}
		We have to show that for every finite set $S \subset \omega$ the identity 
		$\ACL(\ACL(S)) = \ACL(S)$ holds. Let $S \subset \omega$ be an arbitrary finite 
		set and let $x \in \ACL(\ACL(S))$. We will show that $x$ has a finite orbit 
		under $G_{(S)}$ which implies $x \in \ACL(S)$.
		
		It is enough to show that $G_{(S)}(x)$ is finite. Enumerate the elements of $\ACL(S)$ 
		as $\{x_1, x_2 ,\dots, x_k\}$. The group 
		$G_{(S)}$ acts on $\ACL(S)^k$ coordinate-wise. Under this group action 
		the stabiliser of the tuple $(x_1, x_2, \dots, x_k)$ is $G_{(\ACL(S))}$. The 
		Orbit-Stabiliser Theorem states that for any group action the index of the 
		stabiliser of an element in the whole group is the same as the cardinality of 
		its orbit. This yields that the index $[G_{(S)} : G_{(\ACL(S))}]$ is the same 
		as the cardinality of the orbit of $(x_1, x_2, 
		\dots, x_k)$. This orbit is finite because the whole space $\ACL(S)^k$ is 
		finite. So $G_{(\ACL(S))}$ has finite index in $G_{(S)}$.
		
		Let $g_1, g_2, \dots, g_n \in G_{(S)}$ be a left transversal for 
		$G_{(\ACL(S))}$ in $G_{(S)}$, then $G_{(S)} = g_1 G_{ACL(S)} \cup \dots \cup 
		g_n G_{ACL(S)}$. Since $G_{(S)}(x) = g_1 G_{(\ACL(S))}(x) \cup g_2 
		G_{(\ACL(S))}(x) \cup \dots \cup g_n G_{(\ACL(S))}(x)$ is a finite union of 
		finite sets, it must be finite.
	\end{proof}
	
	\begin{lemma}\label{l:ACLtrans}
		The operator $\ACL$ is translation invariant in the following sense: if $S 
		\subset \omega$ is a finite set and $g \in G$ then
		$$\ACL(gS) = g\ACL(S).$$
	\end{lemma}
	\begin{proof}
		Let $x\in \omega$ be an arbitrary element, then
		\begin{multline*}
		x \mbox{ and } y \mbox{ are in the same orbit under } G_{(S)} \Leftrightarrow \\
		\exists h \in G_{(S)} : h(y) = x \Leftrightarrow \exists h \in G_{(S)} : gh(y) 
		= g(x) \Leftrightarrow \\ \exists h \in G_{(S)} : ghg^{-1}(g(y)) 
		= g(x) \Leftrightarrow \exists f \in G_{(gS)} : f(g(y)) = g(x) 
		\Leftrightarrow \\
		g(x) \mbox{ and } g(y) \mbox{ are in the same orbit under } G_{(gS)}.
		\end{multline*}
		So an element $x$ has finite orbit under $G_{(S)}$ if and only if $g(x)$ has finite orbit under $G_{(gS)}$.
	\end{proof}
	
	Now we describe a process to generate a probability measure on $G$, a closed 
	subgroup of $S_\infty$ that has the $FACP$. This probability measure will witness that certain sets are Haar null (see Theorem \ref{t:DM}). 
	
	Our random process will define a permutation $p \in G$ in stages. It depends on 
	integer sequences $(M_{i})_{i \in \omega}$ and $(N_i)_{i \in \omega}$ with 
	$M_i, N_i \ge 1$.
	
	We denote the partial permutation completed in stage $i$ by $p_i$. We start 
	with $p_0 = \id_{ACL(\emptyset)}$ and maintain throughout the following Property 
	\ref{dm:extension} for every $i \ge 1$, and Properties 
	\ref{dm:algebraically closed} and \ref{dm:partial permutation} for $i \in 
	\omega$:
	
	\begin{enumerate}[label=(\roman*)]
		\item\label{dm:extension} $p_{i - 1} \subset p_{i}$,
		\item\label{dm:algebraically closed} $\dom(p_i)$ and $\ran(p_i)$ are finite 
		sets such that $\ACL(\dom(p_i)) = \dom(p_i)$, $\ACL(\ran(p_i) )= \ran(p_i)$,
		\item\label{dm:partial permutation} there is a permutation $g \in G$ that 
		extends $p_i$.
	\end{enumerate}
	
	Let $O_0, O_1, \ldots \subset \omega$ be a sequence of infinite sets with the 
	property that for every finite set $F \subset \omega$ and every infinite orbit 
	$O$ of $G_{(F)}$, the sequence $(O_i)_{i \in \omega}$ contains $O$ infinitely 
	many times. It is easy to see that such a sequence exists, since there exists 
	only countably many such finite sets $F$, and for each one, there exist only 
	countably many orbits of $G_{(F)}$. 
	
	At stage $i \ge 1$, we proceed the following way. First suppose that $i$ is 
	even. We now choose a set $S_i \subset \omega$ with $|S_i| = M_i$ such that 
	$S_i \cap \ran(p_{i - 1}) = \emptyset$. If $i \equiv 0 \pmod{4}$ we require 
	that $S_i$ contains the least $M_i$ elements of 
	$\omega \setminus \ran(p_{i - 1})$, and if $i \equiv 2 \pmod{4}$ we require 
	that $S_i$ contains the least $M_i$ elements of $O_{(i - 2) / 4} \setminus \ran(p_{i - 1})$. 
	Now we will extend $p_{i - 1}$ to a partial permutation $p_i$ such that 
	\begin{equation}
	\label{e:ran(p_i) = ACL(...)}
	\ran(p_i)= \ACL(\ran(p_{i - 1})\cup S_i).
	\end{equation}
	Let us enumerate the elements of $\ACL(\ran(p_{i - 1})\cup S_i) \setminus 
	\ran(p_{i - 1})$ as $(x_1, \dots, x_j)$ such that if $x_1, \dots, x_{k - 1}$ 
	are already chosen then we choose $x_k$ so that 
	\begin{equation}
	\label{e:ordering chosen to be minimal}
	\text{$\ACL(\ran(p_{i - 1}) \cup \{x_1, \ldots, x_k\})$ is minimal with 
		respect to inclusion.}
	\end{equation} 
	
	\begin{claim}
		\label{c:ordering is nice for ran}
		For every $1 \le k \le \ell \le m \le j$, if 
		\begin{equation}
		\label{e:x_m in ACL(..)}
		x_m \in \ACL(\ran(p_{i - 1}) \cup \{x_1, \dots, x_{k}\})
		\end{equation}
		then $$x_\ell \in \ACL(\ran(p_{i - 1}) \cup \{x_1, \dots, x_{k - 1}\} \cup 
		\{x_m\}) \subset \ACL(\ran(p_{i - 1}) \cup \{x_1, \dots, x_{k}\}).$$
		Thus, letting  $\ell=k$ yields  $$ \ACL(\ran(p_{i - 1}) \cup \{x_1, \dots, x_{k - 1}\} \cup 
		\{x_m\}) = \ACL(\ran(p_{i - 1}) \cup \{x_1, \dots, x_{k}\}).$$
	\end{claim}
	\begin{proof}
		The last containment holds, using Lemma \ref{l:idempotentness} and 
		\eqref{e:x_m in ACL(..)}. If $\ell = m$ then there is nothing to prove. Now 
		suppose towards a contradiction that there exists $\ell < m$ violating the 
		statement of the claim, and suppose that $\ell$ is minimal with $k \le \ell 
		< m$ and  
		\begin{equation}
		\label{e:x_ell not in ACL(...)}
		x_{\ell} \notin \ACL(\ran(p_{i - 1}) \cup \{x_1, \ldots, x_{k - 1}\} \cup 
		\{x_m\}).
		\end{equation}
		
		Using the minimality of $\ell$, $\{x_1, \ldots, x_{\ell - 1}\} \subset 
		\ACL(\ran(p_{i - 1}) \cup \{x_1, \ldots, x_{k - 1}\} \cup \{x_m\})$, thus 
		an application of Lemma \ref{l:idempotentness} and the fact that $k \le \ell$ 
		shows that $\ACL(\ran(p_{i - 1}) \cup \{x_1, \ldots, x_{k - 1}\} \cup 
		\{x_m\}) = \ACL(\ran(p_{i - 1}) \cup \{x_1, \ldots, x_{\ell - 1}\} \cup 
		\{x_m\})$. By \eqref{e:x_ell not in ACL(...)} it follows that $x_{\ell} 
		\notin \ACL(\ran(p_{i - 1}) \cup \{x_1, \ldots, x_{\ell - 1}\} \cup 
		\{x_m\})$. Using this, the fact that $k \le \ell$ and \eqref{e:x_m in 
			ACL(..)}, $\ACL(\ran(p_{i - 1}) \cup \{x_1, \ldots, x_{\ell - 1}\} \cup 
		\{x_m\}) \subsetneqq \ACL(\ran(p_{i - 1}) \cup \{x_1, \ldots, x_{\ell}\})$ 
		contradicting \eqref{e:ordering chosen to be minimal}, since $x_{\ell}$ was 
		chosen after $\{x_1, \dots, x_{\ell - 1}\}$ to satisfy that $\ACL(\ran(p_{i - 
			1}) \cup \{x_1, \ldots, x_{\ell}\})$ is minimal. 
	\end{proof}
	
	We will determine the preimages of $(x_1, x_2, \ldots, x_j)$ in this order. 
	Denote the partial permutations defined in these sub-steps by $p_{i, k}$ so 
	that $\ran(p_{i, k}) = \ran(p_{i - 1}) \cup \{x_1, \dots, x_k\}$ for $k = 0, 
	\dots, 
	j$. If the first $k$ preimages are determined then there are two possibilities 
	for $x_{k + 1}$:
	\begin{itemize}
		\item[(a)] The set of possible preimages of $x_{k + 1}$ under $p_{i, k}$ is 
		finite, that is, the set $\{g^{-1}(x_{k + 1}) : g \in G, g \supset p_{i, k}\}$ 
		is finite. Then choose one from them randomly with uniform distribution.
		\item[(b)] The set of possible preimages of $x_{k + 1}$ under $p_{i, k}$ is 
		infinite. Then choose one from the smallest $N_i$ many possible values 
		uniformly. 
	\end{itemize}
	We note that the orbit of $x_k$ under the stabiliser $G_{(\ran(p_{i - 1}))}$ is 
	infinite because $x_k \notin \ran(p_{i - 1}) = \ACL(\ran(p_{i - 1}))$ so 
	\begin{equation}
	\label{e:possibility (b) for x_1}
	\text{possibility (b) must occur for at least $x_1$.}
	\end{equation}

	Let $p_i = p_{i, j}$. Properties \ref{dm:extension} and \ref{dm:partial 
		permutation} obviously hold for $i$. Let $g \in G$ be a permutation with $g 
	\supset p_i$. Now $\ran(p_i) = \ACL(\ran(p_i))$ using \eqref{e:ran(p_i) = 
		ACL(...)} and Lemma \ref{l:idempotentness}. Then $\dom(p_i) = g^{-1} 
	\ran(p_i)$, hence using Lemma \ref{l:ACLtrans}, $\ACL(\dom(p_i)) = \ACL(g^{-1} 
	\ran(p_i)) = g^{-1}\ACL(\ran(p_i)) = g^{-1}\ran(p_i) = \dom(p_i)$, showing 
	Property \ref{dm:algebraically closed}. This concludes the case where $i$ is 
	even. 
	
	If $i$ is odd we let $S_i \subset \omega$ be the set of the least $M_i$ 
	elements of $\omega \setminus \dom(p_{i - 1})$, if $i \equiv 1 \pmod{4}$ 
	and the least $M_i$ elements of $O_{(i - 3) / 4} \setminus \dom(p_{i - 1})$, 
	if $i \equiv 3 \pmod{4}$. We extend $p_{i - 1}$ to a partial permutation $p_i$ such that 
	\begin{equation}
	\label{e:dom(p_i) = ACL(...)}
	\dom(p_i)=\ACL(\dom(p_{i - 1}) \cup S_i).
	\end{equation}
	Again, we enumerate the elements of $\ACL(\dom(p_{i - 1}) \cup S_i) \setminus 
	\dom(p_{i - 1})$ as $(x_1, \dots, x_j)$ such that if $x_1, \dots, x_{k - 1}$ 
	are already chosen then we choose $x_k$ from the rest so that $\ACL(\dom(p_{i - 
		1}) \cup \{x_1, \ldots, x_k\})$ is minimal with respect to inclusion. The proof 
	of the following claim is analogous to the proof of Claim \ref{c:ordering is 
		nice for ran}.
	
	\begin{claim}
		\label{c:ordering is nice for dom}
		For every $1 \le k \le \ell \le m \le j$, $x_m \in \ACL(\dom(p_{i - 1}) 
		\cup \{x_1, \dots, x_{k}\})$ implies $x_\ell \in \ACL(\dom(p_{i - 1}) \cup 
		\{x_1, \dots, x_{k - 1}\} \cup \{x_m\}) \subset \ACL(\dom(p_{i - 1}) \cup 
		\{x_1, \dots, x_{k}\})$. Thus, letting $\ell=k$ yields $ \ACL(\ran(p_{i - 1}) \cup \{x_1, \dots, x_{k - 1}\} \cup 
		\{x_m\}) = \ACL(\ran(p_{i - 1}) \cup \{x_1, \dots, x_{k}\}).$
	\end{claim}
	
	We determine the images of $(x_1, x_2, \ldots, x_j)$ in this order. 
	Denote the partial permutations defined in these sub-steps by $p_{i, k}$ so 
	that $\dom(p_{i, k}) = \dom(p_{i - 1}) \cup \{x_1, \dots, x_k\}$ for $k = 0, 
	\dots, j$. If the first $k$ images are determined then there are two 
	possibilities for $x_{k + 1}$:
	\begin{itemize}
		\item[(a)] The set of possible images of $x_{k + 1}$ under $p_{i, k}$ is 
		finite, that is, the set $\{g(x_{k + 1}) : g \in G, g \supset p_{i, k}\}$ 
		is finite. Then choose one from them randomly with uniform distribution.
		\item[(b)] The set of possible images of $x_{k + 1}$ under $p_{i, k}$ is 
		infinite. Then choose one from the smallest $N_i$ many possible values 
		uniformly. 
	\end{itemize}
	Again, the orbit of $x_k$ under the stabiliser $G_{(\dom(p_{i - 1}))}$ is 
	infinite because $x_k \notin \dom(p_{i - 1}) = \ACL(\dom(p_{i - 1}))$ for every 
	$k$, so possibility (b) must occur for at least $x_1$.
	
	Let $p_i = p_{i, j}$. Again, Properties \ref{dm:extension} and 
	\ref{dm:partial permutation} hold for $i$. Let $g \in G$ be a permutation 
	with $g \supset p_i$. Now $\dom(p_i) = \ACL(\dom(p_i))$ using \eqref{e:dom(p_i) 
		= ACL(...)} and Lemma \ref{l:idempotentness}. Then using Lemma 
	\ref{l:ACLtrans}, $\ACL(\ran(p_i)) = \ACL(g\dom(p_i)) = g\ACL(\dom(p_i)) = 
	g\dom(p_i) = \ran(p_i)$, showing Property \ref{dm:algebraically closed}. This 
	concludes the construction for odd $i$. 
	
	Now let $p = \bigcup_i p_i$. This makes sense using \ref{dm:extension}. 
	
	\begin{claim}
		$p \in G$.
	\end{claim}
	\begin{proof}
		First we show that $p \in S_\infty$. Using \ref{dm:partial permutation}, 
		each $p_i$ is a partial permutation, hence injective. Using 
		\ref{dm:extension}, $p$ is the union of compatible injective functions, 
		hence $p$ is an injective function. 
		It is clear from the construction that $\{0, 1, \dots, i - 1\} \subset 
		\dom(p_{4i}) \cap \ran(p_{4i})$ for every $i$, hence $p \in S_\infty$. 
		
		Using \ref{dm:partial permutation}, we can find an element $g_i \in G$ such 
		that $g_i \supset p_i$. It is clear that $g_i \to p$, and since $G$ is a 
		closed subgroup of $S_\infty$, we conclude that $p \in G$. 
	\end{proof}
	
	The following lemma is crucial in proving that almost every element of $G$ 
	has finitely many finite and infinitely many infinite orbits. 
	
	\begin{lemma}\label{l:badelements}
		Suppose that the parameters of the random process $M_1, \dots, M_{i}$ and 
		$N_1, \dots, N_{i - 1}$ are given along with the numbers $K \in \omega$ and 
		$\varepsilon > 0$. Then we can choose $N_i$ so that for every set $S \subset 
		\omega$ with $|S| = K$, the probability that $S \cap (\dom(p_i) 
		\setminus \dom(p_{i - 1})) \neq \emptyset$ if $i$ is even, or that $S \cap 
		(\ran(p_i) \setminus \ran(p_{i - 1})) \neq \emptyset$ if $i$ is odd, is at 
		most $\varepsilon$.
	\end{lemma}
	\begin{proof}
		We suppose that $i$ is even and prove the lemma only in this case. The proof 
		for the case when $i$ is odd is analogous. 
		
		One can easily see using induction on $i$ that if $M_1, \dots, M_{i - 
			1}$ and $N_1, \dots, N_{i - 1}$ are given then the random process can yield 
		only finitely many different  $p_{i - 1}$ as a result. 
		
		Let $p_{i - 1}$ be one of the possible outcomes, and let $(x_1, x_2, \dots, 
		x_{j})$ denote the elements of $\ACL(\ran(p_{i - 1}) \cup S_i) \setminus 
		\ran(p_{i - 1})$ enumerated in the same order as they appear during the 
		construction. Note that this only depends on $p_{i - 1}$ and $M_i$. Let 
		$a_1$ be the index for which $\ACL(\ran(p_{i - 1}) \cup \{x_1\}) = \ran(p_{i - 1}) \cup \{x_1, \dots, x_{a_1}\}$, such an index exists using Claim 
		\ref{c:ordering is nice for ran}. Hence, for every $m \le a_1$, $x_m \in 
		\ACL(\ran(p_{i - 1}) \cup \{x_1\})$, thus using Claim \ref{c:ordering is nice 
			for ran} again, it follows that 
		\begin{equation}
		\label{e:x_1 in ACL(... x_m)}
		\text{$x_1 \in \ACL(\ran(p_{i - 1}) \cup \{x_m\})$ for every $1 \le m \le a_1$.}
		\end{equation} 
		\begin{claim}
			\label{c:k_m exists}
			For every such $m$, there is a unique positive integer $k_m$ such 
			that if $q$ is an extension of $p_{i - 1}$ with $\ran(q) = 
			\ran(p_{i - 1}) \cup \{x_m\}$ (such that $q \subset g$ for some $g \in G$) 
			then $|\{g^{-1}(x_1) : g \in G, g \supset q\}| = k_m$.
		\end{claim}
		\begin{proof}
			Let $H = G_{(\ran(p_{i - 1}) \cup x_m)}$, then 
			\begin{equation}
			\label{e:k def}
			k = |\{g(x_1) : g \in  H\}| = |\{g^{-1}(x_1) : g \in H\}|
			\end{equation}
			is finite using \eqref{e:x_1 in ACL(... x_m)} and the fact that $H$ is a 
			subgroup. It is enough to show that if $q$ is an extension of $p_{i - 1}$ 
			with $\ran(q) = \ran(p_{i - 1}) \cup \{x_m\}$ then $|\{g^{-1}(x_1) : g \in 
			G, g \supset q\}| = k$. 
			
			Let $g_1, \dots, g_k \in H$ with $g^{-1}_\ell(x_1) \ne g^{-1}_n(x_1)$ if 
			$\ell \ne n$. If $h \in G$ is a permutation with $h \supset q$ then $g_n h 
			\supset q$ for 
			every $1 \le n \le k$. Then using the identity $(g_n h)^{-1}(x_1) = h^{-1} 
			(g_n^{-1}(x_1))$, $(g_\ell h)^{-1}(x_1) \ne (g_n h)^{-1}(x_1)$ if $\ell 
			\ne n$. This shows that $|\{g^{-1}(x_1) : g \in G, g \supset q\}| \ge k$. 
			
			To prove the other inequality, suppose towards a contradiction that there 
			exist $g_1, \dots, g_{k + 1}$ with $g_n \supset q$ for every $n \le k + 1$ 
			and $g_\ell^{-1}(x_1) \neq g_n^{-1}(x_1)$ for every $\ell \neq n$. It is 
			easy to see that $g_n g_1^{-1} \in H$ for every $n$, but the values 
			$(g_ng_1^{-1})^{-1}(x_1) = g_1(g_n^{-1}(x_1))$ are pairwise distinct, 
			contradicting \eqref{e:k def}. Thus the proof of the claim is complete. 
		\end{proof}
		
		Now let $k = \max\{k_2, k_3, \dots, k_{a_1}\}$, if $a_1 \ge 2$, otherwise let 
		$k = 1$.
		
		\begin{claim}
			If $N_i > \frac{kKj}{\varepsilon}$ then for every fixed set $S \subset 
			\omega$ with $|S| = K$ we have $\mathbb{P}(p_i^{-1}(x_m) \in S) < 
			\frac{\varepsilon}{j}$ for every $1 \le m \le a_1$.
		\end{claim}
		\begin{proof}
			This is immediate for $m = 1$, since $k \ge 1$, and the preimage of 
			$x_1$ is chosen from $N_i$ many elements using \eqref{e:possibility (b) for 
				x_1}. Now let $m > 1$, using Claim \ref{c:k_m exists} and the fact that $k 
			\ge k_m$, it follows that for every $y \in \omega$, $|\{g^{-1}(x_1) : g \in 
			G, g \supset p_{i - 1}, g(y) = x_m\}| \le k$, hence for the set $R = \{g^{-1}(x_1) : g \in G, g 
			\supset p_{i - 1}, g^{-1}(x_m) \in S\}$, $|R| \le kK$. In order to be able to 
			extend $p_{i - 1}$ to $p_i$ with $p_i^{-1}(x_m) \in S$, 
			we need to choose $p_i^{-1}(x_1)$ from $R$. Since during the construction 
			of the random automorphism, $p_i^{-1}(x_1)$ is chosen uniformly from a set of size 
			$N_i > \frac{kKj}{\varepsilon}$, we conclude that $\mathbb{P}(p_i^{-1}(x_m) \in S) 
			\le \mathbb{P}(p_i^{-1}(x_1) \in R) \le \frac{|R|}{N_i} < \frac{\varepsilon}{j}$. 
		\end{proof}
		
		For the rest of the proof, we need to repeat the above argument until we 
		reach $j$. If $a_1 < j$, let $a_2$ be the index satisfying $\ACL(\ran(p_{i - 
			1}) \cup \{x_1, \dots, x_{a_1 + 1}\}) = \ran(p_{i - 1}) \cup \{x_1, \dots, 
		x_{a_2}\}$, such an index exists using Claim \ref{c:ordering is nice for ran} 
		as before. Again, we can set a lower bound for $N_i$ so that the for every 
		$a_1 < m \le a_2$, $\mathbb{P}(p_i^{-1}(x_m) \in S) < \frac{\varepsilon}{j}$. 
		Repeating the argument, we can choose $N_i$ so that 
		$\mathbb{P}(p_{i}^{-1}(x_m) \in S) < \frac{\varepsilon}{j}$ for every $1 \le m \le j$, 
		thus $\mathbb{P}(p_i^{-1}(\{x_1, \dots, x_j\}) \cap S \neq \emptyset) < 
		\varepsilon$. Completing the proof of the lemma. 
	\end{proof}
	
	Now we prove a proposition from which our main result will easily follow.
	\begin{proposition}
		\label{p:F and C co-Haar null}
		Let $G \leq S_\infty$ be a closed subgroup. If $G$ has the $FACP$ then the sets 
		\begin{equation*}
		\begin{split}
		\mathcal{F} = \{g \in G :& \text{$g$ has finitely many finite orbits}\}, \\
		\mathcal{C} = \{g \in G :& \forall F \subset \omega \text{ finite } \forall 
		x \in \omega \;(\text{if $G_{(F)}(x)$ is infinite} \\ 
		&\text{then it is not covered by finitely many orbits of $g$})\}  
		\end{split}
		\end{equation*}
		are co-Haar null.  
	\end{proposition}
	
	The set $\mathcal{C}$ could seem unnatural for the first sight. However, from the above fact about the set $\mathcal{C}$ not only our main theorem will be deduced, but this fact also plays a crucial role in proving Theorem \ref{t:randomintro} (see \cite{autrcikk}).

	\begin{proof} 
		We first show the following lemma. 
		\begin{lemma}
			\label{l:conin}
			The sets $\mathcal{F}$ and $\mathcal{C}$ are conjugacy invariant Borel sets.
		\end{lemma}
		\begin{proof}
			The fact that $\mathcal{F}$ is conjugacy invariant follows form the fact 
			that conjugation does not change the orbit structure of a permutation. 
			
			To show that $\mathcal{C}$ is conjugacy invariant, let $c \in 
			\mathcal{C}$, $h \in G$, we need to show that $h^{-1}ch \in \mathcal{C}$. Let $F \subset \omega$ be finite and $x \in \omega$ so that $|G_{(F)}(x)|=\aleph_0$. Note that $G_{(h(F))}(h(x))=hG_{(F)}h^{-1}(h(x))=hG_{(F)}(x)$, hence the first set is also infinite. By $c \in C$ there exists an infinite set $\{x_n: n \in \omega\} \subset G_{(h(F))}(h(x))$ so that for $n \not = n'$ the points $x_n$ and $x_{n'}$ are in different $c$ orbits. But then the points $\{h^{-1}(x_n): n \in \omega\}\subset G_{(F)}(x)$ are in pairwise distinct $h^{-1}ch$ orbits, as desired.

			To show that $\mathcal{F}$ is Borel, notice that the set of permutations 
			containing a given finite orbit is open for every finite orbit. Thus for any 
			finite set of finite orbits the set of permutations containing those finite 
			orbits in their orbit decompositions is open: it can be obtained as the 
			intersection of finitely many open sets. Thus for every $n \in \omega$ the 
			set of permutations containing at least $n$ finite orbits is open: it can 
			be obtained as the union of open sets (one open set for each possible set 
			of $n$ orbits). Thus $S_\infty \setminus \mathcal{F}$ is $G_\delta$: it is 
			the intersection of the above open sets. Hence $\mathcal{F}$ is Borel. 
			
			Now we show that $\mathcal{C}$ is  also Borel. It is enough to show that if 
			$H \subset \omega$ is arbitrary then the set 
			$H^* = \{g \in G : \text{finitely many orbits of $g$ cannot cover $H$}\}$ 
			is Borel, since $\mathcal{C}$ can be written as the countable intersection of such sets. 
			And $H^*$ can be easily seen to be Borel for any $H$, since its complement, 
			$\{g \in G : \exists n \ \forall m \in H \ \exists k \ \exists i < n \ (g^k(i) = m)\}$ 
			is $G_{\delta\sigma}$, hence $H^*$ is $F_{\sigma\delta}$. 
		\end{proof}
		
		To prove the proposition, we use the above construction to generate a random 
		permutation $p$. We set $M_i = 2^i$ for every $i \ge 1$ and we define 
		$(N_i)_{i \ge 1}$ recursively. If $N_1, \dots, N_{i - 1}$ are already 
		defined, then, as before, the random process can yield only finitely many 
		distinct $p_{i - 1}$. Hence, there is a bound $m_i$ depending only on $N_1, 
		\dots, N_{i - 1}$ such that $|\dom(p_i)| = |\ran(p_i)| \le m_i$, since 
		$|\ran(p_i)| = |\ACL(\ran(p_{i - 1}) \cup S_i)|$ if $i$ is even and 
		$|\dom(p_i)| = |\ACL(\dom(p_{i - 1}) \cup S_i)|$ if $i$ is odd, which is 
		independent of $N_i$. Now we use Lemma \ref{l:badelements} to choose $N_i$ so 
		that the conclusion of the lemma is true with $K = m_i$ and $\varepsilon = 
		\frac{1}{2^i}$. 
		
		Using Lemma \ref{l:conjugacyinvariant} and the fact that the sets 
		$\mathcal{F}$ and $\mathcal{C}$ are conjugacy invariant, it is enough to show that 
		\begin{equation}
		\label{e:ph-ban veges sok veges}
		\mathbb{P}(\text{$ph$ has finitely many finite orbits} ) = 1
		\end{equation}
		and
		\begin{equation}
		\label{e:ph-ban veges sok ciklus nem fedi O-t}
		\mathbb{P}(\text{finitely many orbits of $ph$ do not cover $O$} ) = 1
		\end{equation}
		for every $h \in G$, every finite $F \subset \omega$ and every infinite orbit $O$ 
		of $G_{(F)}$, since there exist only 
		countably many such orbits. So let us fix $h \in G$ and an infinite orbit $O 
		\subset \omega$ of $G_{(F)}$ for some finite $F \subset \omega$ for the rest 
		of the proof.
		
		For a partial permutation $q$, a \emph{partial path} in $q$, is a sequence 
		$(y, q(y), \dots, q^n(y))$ with $n \ge 1$, $q^n(y) \notin \dom(q)$ and $y 
		\notin \ran(q)$. Note that $p_ih$ is considered a partial permutation with 
		$\dom(p_ih) = h^{-1}(\dom(p_i))$ and $\ran(p_ih) = \ran(p_i)$. 
		
		During the construction of the random permutation, an \emph{event} occurs 
		when the partial permutation is extended to a new element at some stage 
		regardless of whether it happens for possibility (a) or (b). Suppose that 
		during an event, the partial permutation $p'$ is extended to $p'' = p' \cup 
		(x, y)$. We call this event \emph{bad} if the number of partial paths 
		decreases or $h^{-1}(x) = y$. 
		Note that an event is bad if the extension connects two partial paths of 
		$p'h$ or it completes an orbit (possibly a fixed point). 
		\begin{claim}
			\label{c:dm:1}
			Almost surely, only finitely many bad events happen. 
		\end{claim}
		\begin{proof}
			Let $i$ be fixed and suppose first that it is even. It is easy to see that 
			a bad event can only happen at stage $i$ if a preimage is chosen from 
			$h(\ran(p_{i - 1}))$, that includes the case when a fixed point is 
			constructed. Note that $|\ran(p_i)| \le m_i$, thus the probability of 
			choosing a preimage from this set is at most $\frac{1}{2^i}$, using Lemma 
			\ref{l:badelements}. 
			
			We proceed similarly if $i$ is odd. Then to connect partial paths or 
			complete orbits, an image has to be chosen from the set 
			$h^{-1}(\dom(p_{i - 1}))$. Since  $|\dom(p_i)| \le m_i$, the probability of 
			choosing from this set is at most $\frac{1}{2^i}$. 
			
			Using the Borel--Cantelli lemma, the number of $i$ such that 
			a bad event happens at stage $i$ is finite almost surely. The fact that only a 
			finite number of bad events can happen at a particular stage 
			completes the proof of the claim. 
		\end{proof}
		
		Since a finite orbit can only be created during a bad event, \eqref{e:ph-ban 
			veges sok veges} follows immediately from the claim. Thus $\mathcal{F}$ is 
		co-Haar null.
		
		Now we prove that $\mathcal{C}$ is also co-Haar null by showing 
		\eqref{e:ph-ban veges sok ciklus nem fedi O-t}. Let $n_0, n_1, \ldots \in 
		\omega$ be a sequence with $n_0 < n_1 < \dots $ and $O_{n_i} = O$ for every 
		$i \in \omega$. Let $c_i$ be the number of partial paths of $p_{4n_i + 2}h$ 
		intersecting $O$. It is enough to show that the sequence $(c_i)_{i \in 
			\omega}$ is unbounded almost surely, since using Claim \ref{c:dm:1}, only 
		finitely many of such partial paths can be connected in later stages, hence 
		infinitely many orbits of $ph$ will intersect $O$, almost surely. 
		
		At stage $4n_i + 2$, $p_{4n_i + 1}$ is extended to $p_{4n_i + 2}$ with 
		$\ran(p_{4n_i + 2}) \setminus \ran(p_{4n_i + 1}) \supset S_{4n_i + 2}$, 
		$|S_{4n_i + 2}| = M_{4n_i + 2} = 2^{4n_i + 2}$ and 
		$S_{4n_i + 2} \subset O_{(4n_i + 2-2) / 4} = O_{n_i} = O$. 
		Hence, it is enough to prove that apart from a finite number of exceptions, 
		the elements of $\ran(p_{4n_i + 2}) \setminus \ran(p_{4n_i + 1})$ are in 
		different partial paths in $p_{4n_i + 2}h$, almost surely.
		
		The proof of this fact is similar to the proof of Claim \ref{c:dm:1}. 
		An element $y \in O\cap (\ran(p_{4n_i + 2}) \setminus \ran(p_{4n_i + 1}))$ 
		can only be contained in a completed orbit (of $p_{4n_i + 2}h$), if 
		$h^{-1}p_{4n_i + 2}^{-1}(y) \in \ran(p_{4n_i + 2})$, hence $p_{4n_i + 
			2}^{-1}(y) \in h(\ran(p_{4n_i + 2}))$. Similarly, if $y, y' \in O \cap 
		(\ran(p_{4n_i + 2}) \setminus \ran(p_{4n_i + 1}))$ are in 
		the same partial path (in $p_{4n_i + 2}h$) such that $y$ is the not the first 
		element of this path, then $p_{4n_i + 2}^{-1}(y) \in h(\ran(p_{4n_i + 2}))$. 
		Again using Lemma \ref{l:badelements}, the probability of this happening at 
		stage $4n_i + 2$ is at most $\frac{1}{2^{4n_i + 2}}$, since $|\ran(p_{4n_i + 
			2})| \le m_{4n_i + 2}$. As before, the application of the Borel--Cantelli 
		lemma completes the proof of \eqref{e:ph-ban veges sok ciklus nem fedi O-t}. 
		And thus the proof of the proposition is also complete. 
	\end{proof}
	
	\begin{theorem}\label{t:DM}
		Let $G \leq S_\infty$ be a closed subgroup. If $G$ has the $FACP$ then the sets 
		\begin{equation*}
		\begin{split}
		\mathcal{F} = \{g \in G : \text{$g$ has finitely many finite orbits}\}, \\
		\mathcal{I} = \{g \in G : \text{$g$ has infinitely many infinite orbits}\}  
		\end{split}
		\end{equation*}
		are both co-Haar null. Moreover, if $\mathcal{F}$ is co-Haar null then $G$ 
		has the $FACP$. 
	\end{theorem}
	
	\begin{proof}

		The fact that $\mathcal{F}$ is co-Haar null follows immediately from 
		Proposition \ref{p:F and C co-Haar null}. Let $\mathcal{C}$ denote the set as 
		in Proposition \ref{p:F and C co-Haar null}. If $g \in \mathcal{C}$ 
		then $g$ contains infinitely many orbits, since otherwise finitely many 
		orbits of $g$ could cover $\omega$, hence every infinite orbit of $G_{(F)}$ 
		for some finite $F \subset \omega$. It follows that the co-Haar null set 
		$\mathcal{C} \cap \mathcal{F}$ is contained in $\mathcal{I}$, hence 
		$\mathcal{I}$ is also co-Haar null. And thus the proof of the first part of 
		the theorem is complete. 
		
		Now we prove the second assertion. We have to show that if $G$ does not have 
		the $FACP$ then $\mathcal{F}$ is not co-Haar null. If $G$ does not have the 
		$FACP$ then there is a finite set $S \subset \omega$ such that $\ACL(S)$ is 
		infinite. This means that all of the permutations in $G_{(S)}$ have 
		infinitely many finite orbits, hence $G_{(S)} \cap \mathcal{F} = \emptyset$. 
		The stabiliser $G_{(S)}$ is a non-empty open set, thus it cannot be Haar 
		null. Therefore the proof of the theorem is complete. 
	\end{proof}
	
	Now we are ready to prove the main result of this section.
	
	\begin{theorem}
		\label{t:gen}
		Let $\mathcal{A}$ be a locally finite Fra\"iss\'e limit. Then the following are equivalent:
		\begin{enumerate}
			\item almost every element of $\aut(\mc{A})$ has finitely many finite orbits,
			\item $\aut(\mc{A})$ has the FACP,
			\item $\mc{A}$ has the CSAP.
		\end{enumerate}
		Moreover, any of the above conditions implies that almost every element of $\mc{A}$ has infinitely many infinite orbits.
	\end{theorem}
	
	\begin{proof}
		
		The equivalence \eqref{tc:finfin} $\iff$ \eqref{tc:FACP}, and the last statement of the theorem is just the application of Theorem \ref{t:DM} to $G=\aut(\mc{A})$. Thus, it is enough to show that \eqref{tc:FACP} $\iff$ \eqref{tc:CSAP}. 
		
		Let $\mc{K}=age(\mc{A})$. Since $\mc{A}$ is the limit of $\mc{K}$, using that $\mc{A}$ is ultrahomogeneous it follows that $\mc{K}$ has the extension property, that is, for every $\mc{B},\mc{C} \in \mc{K}$ and embeddings $\phi:\mc{B} \to \mc{C}$ and $\psi:\mc{B} \to \mc{A}$ there exists an embedding $\psi':\mc{C} \to \mc{A}$ with $\psi' \circ \phi=\psi$. Thus, the embeddings between the structures in $\mc{K}$ can be considered as partial automorphisms of $\mc{A}$. 
		
		(\eqref{tc:FACP} $\Rightarrow$ \eqref{tc:CSAP}) Take an arbitrary $\mc{B}_0 \in \mc{K}$ and fix an isomorphic copy of it inside $\mc{A}$. Let $\mc{B}=\ACL(\dom(\mc{B}_0))$ and note that by the fact that $\aut(\mc{A})$ has the $FACP$  $\mc{B}$ is a finite substructure of $\mc{A}$. We will show that over $\mc{B}$ the strong amalgamation property holds (see Definition \ref{CSAPdef}). In order to see this, let $\mc{C}, \mc{D} \in \mc{K}$ and let $\psi: \mc{B} \to \mc{C}$ and $\phi:\mc{B} \to \mc{D}$ be embeddings. By the extension property we can suppose that $\mc{B}<\mc{C} < \mc{A}$, $\mc{B} < \mc{D}<\mc{A}$ and $\psi=\phi=id_{\mc{B}}$. By Lemma \ref{l:idempotentness} $\ACL(\dom(\mc{B}))=\mc{B}$, hence the $\aut(\mc{A})_{(\dom(\mc{B}))}$ orbit of every point in $\dom(\mc{C})\setminus \dom(\mc{B})$ is infinite. By M. Neumann's Lemma \cite[Corollary 4.2.2.]{hodges} $\dom(\mc{C})\setminus \dom(\mc{B})$ has infinitely many pairwise disjoint copies under the action of $\aut(\mc{A})_{(\dom(\mc{B}))}$. In particular, by the pigeonhole principle, there exists an $f \in \aut(\mc{A})_{(\dom(\mc{B}))}$ such that $f(\mc{C}) \cap \mc{D}=\mc{B}$. Letting $\mc{E}$ to be the substructure of $\mc{A}$ generated by $\dom(f(\mc{C}))\cup \dom(\mc{D})$, $\psi'=f|_{\mc{C}}$ and $\phi'=\id_{\mc{D}}$ shows that SAP holds over $\mc{B}$ and hence CSAP holds as well. 
		
		(\eqref{tc:FACP} $\Leftarrow$ \eqref{tc:CSAP}) Let $S \subset \dom(\mc{A})$ be finite. Let $\mc{B}_0$ be the substructure generated by $S$. Clearly, $\mc{B}_0 \in \mc{K}$, hence there exists a $\mc{B} \in \mc{K}$ over which the strong amalgamation property holds and which contains an isomorphic copy of $\mc{B}_0$. By the extension property of $\mc{A}$ we can suppose that $\mc{B}$ and all the structures constructed later on in this part of the proof are substructures of $\mc{A}$ containing $\mc{B}_0$. 
		
		We claim that for every $b \in \dom(\mc{A}) \setminus \dom(\mc{B})$ the orbit $\aut(\mc{A})_{(\dom(\mc{B}))}(b)$ is infinite. Indeed, let $\mc{C}$ be the substructure generated by $\dom(\mc{B}) \cup \{b\}$. Using the strong amalgamation property repeatedly, first for $\mc{B},\mc{C}$ and $\mc{D}=\mc{C}$ obtaining an $\mc{E}_1$, then for  $\mc{B},\mc{C}$ and $\mc{D}=\mc{E}_1$ obtaining an $\mc{E}_2$ etc. for every $n$ we can find a substructure $\mc{E}_n$ of $\mc{A}$ which contains $n+1$ isomorphic copies of $\mc{C}$ which intersect only in $\mc{B}$, and the isomorphisms between these copies fix $\mc{B}$. Extending the isomorphisms to automorphisms of $\aut(\mc{A})$ shows that the orbit $\aut(\mc{A})_{(\dom(\mc{B}))}(b)$ is infinite.
	\end{proof}
	
	\begin{remark}
		It is not hard to construct countable Fra\"iss\'e classes to show that CSAP is neither equivalent to SAP, nor to AP. An example showing that CSAP $\not \Rightarrow $ SAP is $age(\mathcal{B}_\infty)$. Indeed, using a result of Schmerl \cite{SAP-NoAlg} that states that a Fra\"iss\'e class has the SAP if and only if its automorphism group has no algebraicity (that is, $\ACL(F) = F$ for every finite $F$), $age(\mathcal{B}_\infty)$ cannot have the SAP. 
		
		To see that AP $\not \Rightarrow $ CSAP, let $\mathcal{Z}$ be the structure on the set $\mathbb{Z}$ of integers with a relations $R_n$ for each $n \ge 1$, $n \in \N$ satisfying that $a R_n b \Leftrightarrow |a - b| = n$ for each $a, b \in \mathbb{Z}$ and $n \ge 1$. It can be easily checked that $age(\mathcal{Z})$ satisfies AP, but $\aut(\mathcal{Z})$ does not satisfy FACP, since the algebraic closure of any two points is $\mathbb{Z}$. Thus Theorem \ref{t:gen} implies that $\mathcal{Z}$ cannot satisfy CSAP.
		
	\end{remark}

	\section{An application to decompositions}
	\label{s:appl}
	In this section we present an application of our results: we use Theorem \ref{t:DM} to show that a large family of automorphism groups of countable structures can be decomposed into the union of a Haar null and a meager set. 
	
	\begin{corollary}\label{c:inffixed -> decomp}
		Let $G$ be a closed subgroup of $S_\infty$ satisfying the $FACP$ and suppose that the set $F = \{g \in G : \text{$\fix(g)$ is infinite}\}$ is dense in $G$. Then $G$ can be decomposed into the union of an (even conjugacy invariant) Haar null and a meager set. 
	\end{corollary}
	\begin{proof}
		Clearly, $F$ is conjugacy invariant, and since it can be written as $F = \{g \in G : \forall n \in \omega \ \exists m > n \ \left(g(m) = m\right) \}$, $F$ is $G_\delta$. Using the assumptions of this corollary, it is dense $G_\delta$, hence co-meager. Using Theorem \ref{t:DM}, it is Haar null, hence $F \cup (G \setminus F)$ is an appropriate decomposition of $G$. 
	\end{proof}	
	
	\begin{corollary}
		\label{c:Aut(R), Aut(Q), Aut(B_inf) decomposes}
		$\aut(\mathcal{R})$, $\aut(\mathbb{Q}, <)$ and $\aut(\mc{B}_\infty)$ can be decomposed into the union of an (even conjugacy invariant) Haar null and a meager set.
	\end{corollary}
	\begin{proof}
		In order to show that the set of elements in these groups with infinitely many fixed points is dense, in each case it is enough to show that if $p$ is a finite, partial automorphism then there is another partial automorphism $p'$ extending $p$ such that $p' \supset p \cup (x,x)$ with $x \not \in \dom(p)$. 
		
		For $\aut(\Q, <)$, let $x$ be greater than each element in $\dom(p) \cup \ran(p)$, then it is easy to see that $p \cup (x, x)$ is also a partial automorphism. 
		
		For $\aut(\mathcal{R})$, let $x$ be an element different from each of $\dom(p) \cup \ran(p)$ with the property that $x$ is connected to every vertex in $\dom(p) \cup \ran(p)$. Then it is easy to see that $p \cup (x, x)$ is a partial automorphism. 
		
		For $\aut(\mathcal{B}_\infty)$, let $a_0 \cup a_1 \cup \dots \cup a_{n-1}$ be a partition of $\mathbf{1}$ with the property that $\dom(p) \cup \ran(p)$ is a subset of the algebra generated by $A = \{a_0, a_1, \dots, a_{n - 1}\}$. Then there is a permutation $\pi$ of $\{0, 1, \dots, n - 1\}$ compatible with $p$, that is, $p(a_i) = a_{\pi(i)}$ for every $i$. Let us write each $a_i$ as a disjoint union $a_i = a_i' \cup a_i''$ of non-zero elements. Again, a partial permutation can be described by a permutation of the elements $\{a_1', \dots, a_n'\} \cup \{ a_1'', \dots, a_n''\}$. Hence, let $p'$ be defined by $p'(a_i') = a_{\pi(i)}'$, $p'(a_i'') = a_{\pi(i)}''$. Then $p'$ is a partial automorphism extending $p$ with a new fixed point $\bigcup_{i < n} a_i'$. 
	\end{proof}

	\section{Various behaviors}
	\label{s:poss}
	It turns out, that in natural Polish groups we may encounter very different behaviors of conjugacy classes with respect to the ideal of Haar null sets (see \cite{homeo}, \cite{autrcikk}, \cite{autqcikk}). In this section we address the questions from \ref{q:how}, namely, given a Polish group, how many non-Haar null conjugacy classes are there and decide whether the union of the Haar null classes is Haar null. Note that these questions make perfect sense in the locally compact case as well. In this section we construct a couple of examples. 
	
	If $(A,+)$ is an abelian group we will denote by $\phi$ the automorphism of $A$ defined by $a \mapsto -a$. 
	
	\begin{proposition}
		Let $(A,+)$ be an abelian Polish group such that for every $a \in A$ there exists an element $b$ with $2b=a$. Observe that  $\phi \in \aut(A)$, $\phi^2=id_A$ and $(\Z_2 \ltimes_\phi A,\cdot)$ can be partitioned into $\{0\} \times A$ and $\{1\} \times A$. Moreover, in the group $\Z_2 \ltimes_\phi A$ the conjugacy class of every element of $\{0\} \times A$ is of cardinality at most $2$, whereas the set $\{1\} \times A$ is a single conjugacy class. 
	\end{proposition}
	
	\begin{proof}
		Let $(0,a) \in \{0\} \times A$ and $(i,b) \in \Z_2 \ltimes_\phi A$ arbitrary. We claim that the conjugacy class of $(0,a)$ is $\{(0,a),(0,-a)\}$. If $i=0$ then $(0,a)$ and $(i,b)$ commute, so let $i=1$. By definition 
		\[(1,b)^{-1} \cdot (0,a) \cdot (1,b)=(1,b) \cdot (1,b+a)=(0,b-(b+a))=(0,-a),\]
		which shows our claim. 
		
		Now let $(1,a),(1,a') \in \Z_2 \ltimes_\phi A$ be arbitrary. Now for an arbitrary element $(1,b)$ we get
		\[(1,b)^{-1} \cdot (1,a) \cdot (1,b)=(1,b) \cdot (0,-b+a)=(1,b-(-b+a))=(1,2b-a),\]
		thus, choosing $b$ so that $2b=a'+a$ we obtain
		\[(1,b)^{-1} \cdot (1,a) \cdot (1,b)=(1,a').\]	
	\end{proof}
	\begin{corollary}
		\label{c:semidir}
		Let $A=\Z^\omega_3$ or $A=(\Q_d)^\omega$, (that is, the countable infinite power of the rational numbers with the discrete topology). Then $\Z_2 \ltimes_\phi A$ has a non-empty clopen conjugacy class, namely $\{(1,a):a \in A\}$ and every other conjugacy class has cardinality at most $2$. Hence, the union of the Haar null classes  $\{(0,a):a \in A\}$ is also non-empty clopen.
	\end{corollary}
	\begin{lemma}
		\label{l:prod}
		Suppose that $G_1$ and $G_2$ are Polish groups and $A_1 \subset G_1$ is Borel and $U \subset G_2$ is non-empty and open. Then $A_1 \times U$ is Haar null in $G_1 \times G_2$ iff $A_1$ is Haar null. 
	\end{lemma}
	\begin{proof}
		Suppose first that $A_1$ is Haar null witnessed by a measure $\mu_1$. Then, if $\mu'$ is the same measure copied to $G_1 \times \{1\}$, it is easy to see that $\mu'$ witnesses the Haar nullness of $A_1 \times G_2$, in particular, the Haar nullness of $A_1 \times U$.
		
		Conversely, suppose that $A_1 \times U$ is Haar null witnessed by the measure $\mu$. Clearly, as countably many translates of $U$ cover $G_2$, countably many translates of $A_1 \times U$ cover $A_1 \times G_2$, hence $A_1 \times G_2$ is Haar null as well, and this is also witnessed by the measure $\mu$. Let $\mu_1=\proj_{G_1*}\mu$, then $\mu_1$ witnesses the Haar nullness of $A_1$.
	\end{proof}	
	\begin{proposition}
		\label{p:prod}
		If $G$ is a Polish group with $\kappa$ many non-Haar null conjugacy classes then $G \times \left(\Z_2 \ltimes_\phi  \Z^\omega_3 \right)$ has $\kappa$ many non-Haar null conjugacy classes and the union of the Haar null conjugacy classes is not Haar null.
	\end{proposition}
	\begin{proof}
		Clearly, the conjugacy classes of $G \times \left(\Z_2 \ltimes_\phi  \Z^\omega_3 \right)$ are of the form $C_1 \times C_2$ where $C_1$ is a conjugacy class in $G$ and $C_2$ is a conjugacy class in $\Z_2 \ltimes_\phi  \Z^\omega_3$. By Corollary \ref{c:semidir} we have that every conjugacy class in the latter group is finite with one exception, this exceptional conjugacy class is clopen; let us denote it by $U$. Now, since the finite sets are Haar null in $\Z_2 \ltimes_\phi  \Z^\omega_3$ by Lemma \ref{l:prod}, the set of non-Haar null conjugacy classes in $G \times \left(\Z_2 \ltimes_\phi  \Z^\omega_3\right)$ is equal to $\{C \times U:C \text{ is a non-Haar null conjugacy class in } G\}$, hence the cardinality of the non-Haar null classes is $\kappa$. Moreover, the union of the Haar null conjugacy classes contains $G \times ((\Z_2 \ltimes_\phi  \Z^\omega_3) \setminus U)$, which is non-empty and open, consequently it is not Haar null.
		
	\end{proof}
	Finally, we would like to recall the following well known theorem.
	\begin{theorem}[HNN extension, \cite{higman1949embedding}]\label{t:hnn} There exists a countably infinite group with two conjugacy classes. 
	\end{theorem}
	We denote such a group by HNN, and consider it as a discrete Polish group. 
	
	Combining Proposition \ref{p:prod}, Corollaries \ref{c:autqcont}, \ref{c:autrcont}, \ref{c:semidir}, Lemma \ref{l:prod} and Theorems \ref{t:DMold} and \ref{t:hnn} we obtain Table \ref{tab:1} (see the end of Section \ref{s:results}).  (Recall that $C$, $LC \setminus C$ and $NLC$ stand for compact, locally compact non-compact,  and non-locally compact, respectively.)

	\section{Open problems}
	\label{s:open}
	We finish with a couple of open questions. In Section \ref{s:poss} we produced several groups with various numbers of non-Haar null conjugacy classes. However, our examples are somewhat artificial. 
	
	\begin{question} Are there natural examples of automorphism groups with given cardinality of non-Haar null conjugacy classes?
	\end{question}
	
	The following question is maybe the most interesting one from the set theoretic viewpoint. 
	\begin{question} Suppose that a Polish group has uncountably many non-Haar null conjugacy classes. Does it have continuum many non-Haar null conjugacy classes?
		
	\end{question}
	
	The answer is of course affirmative under e.g.~the Continuum Hypothesis.
	Since the definition of Haar null sets is complicated (the collection of non-Haar null closed sets can already be $\mathbf{\Sigma}^1_1$-hard and $\mathbf{\Pi}^1_1$-hard \cite{openlyhaarnull}), it is unlikely that this question can be answered with an absoluteness argument. 
	
	The characterization result of Section \ref{s:genres} and the similarity between Theorems \ref{t:qintro} and \ref{t:randomintro} suggest that a general theory of the behavior of the random automorphism (similar to the one built by Truss, Kechris and Rosendal) could exist. 
	
	\begin{problem} Formulate necessary and sufficient model theoretic conditions which characterize the measure theoretic behavior of the conjugacy classes.
	\end{problem}
	
	In particular, it would be very interesting to find a unified proof of the description of the non-Haar null classes of $\aut(\Q, <)$ and $\aut(\mc{R})$. 
	
	\bigskip
	\textbf{Acknowledgements.} We would like to thank to R. Balka, Z. Gyenis, A. Kechris, C. Rosendal, S. Solecki and P. Wesolek for many valuable remarks and discussions. We are also grateful to the anonymous referee for their comments and suggestions, particularly for pointing out a simplification of the proof of Lemma \ref{l:conin}.
	
	\bibliographystyle{apalike}
	
	\bibliography{ran}

\end{document}